\documentclass{amsart}

\usepackage{amsmath, amsthm, amssymb, graphicx, enumerate, everypage, datetime, getfiledate, lmodern, xcolor}

\newcounter{forwardtheorem}

\newcounter{citedtheorems}

\newtheorem{defn}{Definition}[section]
\newtheorem{theorem}[defn]{Theorem}
\newtheorem{prop}[defn]{Proposition}
\newtheorem*{theorem-x}{Theorem}
\newtheorem*{theorem-m}{Main Theorem}
\newtheorem*{theorem-n}{Main Theorem}
\newtheorem*{cor-x}{Corollary}
\newtheorem*{lemma-x}{Lemma}
\newtheorem*{concl-x}{Conclusion}
\newtheorem*{claim-x}{Claim}
\newtheorem*{thm-r}{Theorem \ref{concl:sop2-max}}
\newtheorem*{thm-q}{Theorem \ref{theorem:p-t}}
\newtheorem*{claim-s}{Claim \ref{m1-sat}}

\newtheorem{thm-lit}[citedtheorems]{Theorem}
\newtheorem{defn-a}[citedtheorems]{Definition}

\newtheorem{fact}[defn]{Fact}
\newtheorem{cor}[defn]{Corollary}

\newtheorem{concl}[defn]{Conclusion}
\newtheorem{conv}[defn]{Convention}
\newtheorem{conv-r}[defn]{Conventions and Remarks}
\newtheorem{claim}[defn]{Claim}

\newtheorem{lemma}[defn]{Lemma}
\newtheorem{obs}[defn]{Observation}
\newtheorem{rmk}[defn]{Remark}
\newtheorem*{rmk-star}{Remark}

\newtheorem{disc}[defn]{Discussion}

\newcommand{\sbr}{\vspace{1mm}}

\newtheorem{expl}[defn]{Example}

\newcommand{\de}{\mathcal{D}}

\newcommand{\mx}{\mathbf{x}}
\newcommand{\my}{\mathbf{y}}

\newcommand{\possible}[1]{{\color{blue}}{\color{black}#1}}

\newcommand{\vp}{\varphi}
\newcommand{\uu}{\mathcal{U}}

\newcommand{\br}{\vspace{4mm}}

\newcommand{\mbp}{\mathbf{P}}
\newcommand{\tlf}{\lesssim}
\newcommand{\mcr}{\mathcal{R}}
\newcommand{\tlg}{\tilde{\Gamma}}

\newcommand{\hn}{H_{2n+1}}

\newcommand{\hip}{H_\infty(p)}
\newcommand{\id}{\operatorname{id}}
\newcommand{\supp}{\operatorname{sup}}

\newcommand{\mcx}{\mathbf{X}}
\newcommand{\mcy}{\mathbf{Y}}
\newcommand{\UT}{\operatorname{UT}}
\newcommand{\GL}{\operatorname{GL}}

\title[Complexity and randomness in the Heisenberg groups]{Complexity and randomness \\ in the Heisenberg groups (and beyond) }

\date{Version of August 13, 2021. }

\author[Diaconis]{Persi Diaconis}
\address{Department of Mathematics and Statistics, Stanford University}
\email{diaconis@math.stanford.edu}

\author[Malliaris]{Maryanthe Malliaris}
\address{Department of Mathematics, University of Chicago}
\email{mem@math.uchicago.edu}

\begin{document}

\maketitle

\begin{abstract}
By studying the commuting graphs of conjugacy classes of the sequence of Heisenberg groups $H_{2n+1}(p)$ and their limit $H_\infty(p)$ 
we find pseudo-random behavior (and the random graph in the limiting case). This makes a nice case study for transfer of 
information between finite and infinite objects. 
Some of this behavior transfers to the problem of understanding what makes understanding the character theory of 
the uni-upper-triangular group (mod p) ``wild.'' 
Our investigations in this paper may be seen as a meditation on the question: is randomness simple or is it complicated? 
\end{abstract}

\vspace{10mm}
\hfill{\emph{In memoriam Vaughan Jones.}}

\vspace{10mm}

This paper is dedicated to Vaughan Jones. 
Both of the authors knew Vaughan, both from Berkeley and through being invited speakers to the seminars he organized in New Zealand, 
in Kaikoura and Hanmer Springs, and in Napier, respectively. 
He was a wonderful, lively human being. Always trying to communicate and get you involved with whatever current project he had in focus -- 
from wind surfing to the the mathematics of tangles. He was happy to explain anything he knew about `in English'. 

He also liked to keep 
things lively. For example, during visits to Kaikoura, Diaconis and Vaughan's great friend Hugh Woodin fell in love with a New Zealand chocolate bar 
`Peanut Slabs'. After the neighborhood supply was exhausted, they made forays to nearby towns surprised to find ``somebody came in and bought out our 
stock yesterday''.  Vaughan had spotted their weakness and cornered local supplies. He walked through the conference the next day 
preening (and munching on a peanut slab (sigh)). He shared later, but the story went on. Diaconis routinely gives five or so talks a year at Berkeley. For 
the NEXT SEVERAL YEARS, walking into the conference room, there was a peanut slab on the lecturers table. Whether it was 
Vaughan or Hugh, it was a wonderful inducement. 

\vspace{5mm}

\noindent \textbf{Acknowledgments.}
Thanks to Peter Cameron, Bob Guralnick, Marty Isaacs, Justin Moore, Benny Sudakov, 
and Nat Thiem; and 
thanks to the organizers and participants of the New Zealand summer schools in 
Hanmer Springs, Kaikoura, and Napier.  
\possible{We thank the referee for very helpful comments. } 
Diaconis was partially supported by NSF 1954042, and Malliaris 
by NSF 1553653.

\newpage

\section{Introduction}

Let $n \in \mathbb{N}$ and $p$ be a prime. 
Define the Heisenberg group $\hn(p)$ as $\{ [x,y,z] : x,y \in \mathbb{F}^n_p, z \in \mathbb{F}_p \}$ 
with $[x,y,z][x^\prime, y^\prime, z^\prime] = 
[x+x^\prime, y+y^\prime, z+z^\prime + xy^\prime]$ where $xy^\prime = x_1 y^\prime_1 + \cdots + x_n y^\prime_n$. 
Thus $\id = [0,0,0]$, $[x,y,z]^{-1} = [-x,-y,-z+x y]$, $[x,y,z]^{-1}[a,b,c] [x,y,z] = [a, b, c +ay -xb ]$, $[x,y,z]^{-1}[a,b,c]^{-1}[x,y,z][a,b,c] = [0,0,xb-ay]$.  
Thus $[x,y,z]$ and $[x^\prime, y^\prime, z^\prime]$ commute if and only if $xy^\prime - yx^\prime = 0 ~(\operatorname{mod} p)$. 
As discussed below, $\hn(p)$ are the ``extraspecial $p$-groups of exponent $p$''.  Thus $|\hn(p)| = p^{2n+1}$, 
$Z(\hn(p)) = H^\prime_{2n+1}(p) = \Phi(\hn(p)) = \{ [0,0,z] : z \in \mathbb{F}_p \} \cong \mathbb{F}_p$, and (for $p>2$) $[x,y,z]^p = \id$. 
Here $\Phi$ is the Frattini subgroup,  
$H^\prime$ is the commutator subgroup, and $Z(H)$ is the center. 
These definitions and facts hold as well for $H_\infty(p)$ with $\mathbb{F}^n_p$ replaced by countably
 infinite vectors with at most finitely many nonzero elements. 
$H_{2n+1}(p)$ can also be seen as the $(n+2)\times(n+2)$ matrices with ones on the main diagonal, zeroes below the diagonal, $x$ in the top row
(i.e., second through $(n+1)$-st entries of the top row), 
$y$ in the last column (second through $(n+1)$-st entries) and $z$ in the upper right hand corner. 

The conjugacy classes of $\hn(p)$ and $\hip$ are the center and the cosets of the center. Let $[x,y, *] = C_{xy}$ denote these non central classes. 
Form a graph $\Gamma(H_{2n+1})$  (respectively $\Gamma(H_{\infty})$)  
with vertices the elements of $H_{2n+1}(p)$ and an edge between two elements if they commute; such a $\Gamma$ is often called a commuting graph.  
[For some purposes we may want to include a loop at each vertex, because an element commutes with itself.]   
Observe that if $[x,y,z] \in C_{xy}$ and $[a,b,c] \in C_{ab}$ commute then all elements of $C_{xy}$ commute with all elements of $C_{ab}$. 
Thus, in $\Gamma$, each conjugacy class $C_{xy}$ forms a complete graph, and the induced bipartite graph between any two distinct conjugacy classes 
$C_{xy}$ and $C_{ab}$ is either complete or empty. So we may also 
form the quotient graph $\tilde{\Gamma}(H_{2n+1})$ (respectively $\tilde{\Gamma}(H_{\infty})$) with 
vertices the non-central conjugacy classes and an edge from $C_{xy}$ to $C_{ab}$ if their elements commute, that is, if $xb-ay = 0$.   
A subtlety to note in this definition is that the quotient group mod the center is elementary abelian.  

It will be useful to know that:

\begin{fact}[\cite{e-t}, see \S \ref{s:c-g} below] \label{et-fact} For any finite group $G$, 
\[ e(G) = |G| c(G) \]
where $e(G)$ is the number of ordered pairs of commuting elements, 
$|G|$ is the size of $G$, and $c(G)$ is the number of conjugacy classes. 
\end{fact}

In the case of $H_{2n+1}(p)$, we have seen  
the conjugacy classes are determined by $2n$ entries; subtract the all-zero entry, and add the $p$ central elements.  So Fact \ref{et-fact} gives 
$e(H_{2n+1}) =  |H_{2n+1}| c(H_{2n+1}) = p^{2n+1} (p^{2n}+p-1)$. 
Dividing both sides by $p^{2(2n+1)}$ shows that if two elements of $H_{2n+1}(p)$ are chosen at random, the chance that they commute is about 
$1/p + (1/p^{2n})$. 

Section two contains background material on extraspecial $p$-groups, quasirandomness, `the' random graph, axioms for $H_\infty(p)$, 
symplectic spaces, and commuting graphs. 
In section three we give a `bare hands' proof that the commuting graphs of the Heisenberg groups $H_{2n+1}(p)$ 
are quasi-random when $n$ is large for fixed 
$p$. Section four shows that $H_\infty$ contains the random graph as an induced subgraph in a strong sense (it is essentially generated by it) 
and so enjoys many parallel properties of randomness. Section five connects these random properties of the Heisenberg group to the group 
$UT(n,p)$ -- uni-upper-triangular matrices with entries in $\mathbb{F}_p$. It begins with a separate literature review on the difficulties of describing the 
conjugacy classes or characters of $UT$. This may be read now for further motivation.


\br

\section{Background}

This section gives background material on Heisenberg and extraspecial groups (2.1), the random graph (2.2), quasirandomness (2.3), axioms (2.4), 
\possible{symplectic spaces (2.5),} and commuting graphs (2.6). 
Since this is an interdisciplinary effort we have made an effort to bring the various topics to life (and to articulate one new open problem). 

Note: in our main cases below, $p$ will be an odd prime. 

\subsection{Heisenberg and extra-special $p$-groups}

The construction of $\hn$ and $H_\infty$ given above makes sense for any ring. Over $\mathbb{R}$, the Heisenberg groups form a central part of analysis 
(if this seems like overselling, see \cite{hbook}). Over $\mathbb{Z}$, they are a basic ingredient of theta functions in many variables. See \cite{mumford}. We will work over the 
finite field $\mathbb{F}_p$. A finite group $G$ of order $p^N$ is extraspecial if $Z(G) = G^\prime = \Phi(G) \cong \mathbb{F}_p$.  Extraspecial $p$-groups were introduced by 
Philip Hall and had early success in the Hall-Higman paper on Burnside's problem \cite{h-h}. They are a standard topic in group theory texts 
\cite{aschbacher} \cite{suzuki} \cite{huppert}.  It is known that $N = 2n+1$ is forced.


There are two nonisomorphic such groups (fixing $n$, $p$): 
our Heisenberg groups $H_{2n+1}(p)$ 
(distinguished by having every element of order at most $p$) and the groups 
$M_{2n+1}(p)$ which have elements of order $p^2$.\footnote{\possible{There are also two extraspecial groups of order $2^{2n+1}$, but neither of them has 
exponent 2.}} For example $M_3(p)$ may be constructed as a semi-direct product, letting 
$C_p$ act on $C(p^2)$ by $j\cdot k = (1+jp)k \mod p^2$. $M_{2n+1}(p)$ may be constructed as a central product of $H_{2n-1}$ and $M_3(p)$. 
So the present constructions for $H_{2n+1}$ may be easily mirrored for $M_{2n+1}(p)$. Indeed, $M_{2n+1}(p)$ 
and $H_{2n+1}$ have the same character table. 
N.B. while there is only one $H_{2n+1}$ and only one countable $H_\infty$ there are two distinct countable $M_\infty$'s \cite{hall}.\footnote{corresponding to different completions of the theory.}

The Heisenberg groups occur throughout mathematics and physics; there is even a literature on probability for these groups. 
To explain, work in $H_3(p)$. A minimal generating set is $S = \{ [1,0,0], [-1,0,0], [0,1,0], [0,-1,0] \}$. A Markov chain on 
$H_3(p)$ based on this is:
\[ K(x,y) = 
\begin{cases}
\frac{1}{5} & \text{if } yx^{-1} \in S \cup \{ \id \} \\
0 & \text{otherwise.}
\end{cases}
\]
The work cited below shows that this Markov chain has a uniform stationary distribution and order $p^2$ steps are necessary and 
sufficient for convergence to stationarity. The proofs use the eigenvalues and eigenfunctions of the associated Laplace operator $L(x,y) = 
\mathcal{I} - K(x,y)$. The spectrum of the  Laplacian is a basic ingredient in understanding the geometry of the underlying space. 
See \cite{jorgenson}, \cite{liu}, and \cite{bump}, \cite{dh} which contain reviews.

In logic, the extraspecial $p$-groups have also been present for several decades; we mention a few central points. 
First, suppose $p \neq 2$.  
\cite{felgner} gave axioms $T_p$ for $H_\infty(p)$ and proved there is only one countable such group, up to isomorphism, see \S 2.4 below. 
It is easy to see that the infinite extraspecial $p$-groups 
are unstable in the sense of Shelah's classification theory \cite{shelah}: it suffices to find a formula  
$\vp(x,y)$ and disjoint sets of distinct elements $\{ a_i : i < \omega \}$, $\{ b_j : j < \omega \}$ such that $\vp(a_i, b_j)$ holds if and only if 
$i<j$. (Use the formula $[x,y]=1$ -- conveniently, our edge relation in $\Gamma$ -- and section three, or 
read the footnote on page \pageref{footnote-4}.) 
Thus, as already observed by Felgner, Shelah's theory applies to show 
that there are $2^\lambda$ pairwise nonisomorphic extraspecial $p$-groups of exponent $p$ of 
each uncountable size $\lambda$; a more direct construction is in \cite{sh-st}.   
(Some parallel results on the case of exponent $p^2$ are in \cite{hall}.)
The extraspecial $p$-groups are a useful example of a (very) simple unstable theory -- 
here ``simple'' can be read as both a model theoretic term and an 
English adjective -- especially in their alternate 
guise as symplectic spaces over finite fields. 
See the work of \cite{c-h}, and \cite{mac-st}.
In yet another direction, they were used by \cite{sh-st} to build so-called Ehrenfeucht-Faber groups, 
uncountable groups with every abelian subgroup of strictly smaller cardinality.

We conclude this discussion by stating what might be the first appearance of extraspecial $p$-groups.
\begin{theorem}[P. Hall; see Suzuki p. 75]
For prime $p>2$, let $G$ be a $p$-group with every characteristic abelian subgroup cyclic. Then $G$ is a central product of $H$ and $S$ where $H$ is a 
Heisenberg group and $S$ is cyclic. [N.B. $H$ may equal $\id$.] 
\end{theorem}

\subsection{The random graph}

This remarkable object appears in graph theory and logic. There is a wonderful survey, with many readable, full proofs, in Peter Cameron's fine \cite{cameron}. 
As motivation, fix a large $n$, say 100. Pick two Erd\H{o}s-Renyi $\frac{1}{2}$-graphs randomly. What is the chance they are isomorphic? Intuitively, ``small''. 
How small? Well, the number of possible one to one maps from $[n]$ to $[n]$ is $n!$ and the number of graphs is $({2})^{\binom{n}{2}}$ so it is at most 
$n!/2^{\binom{n}{2}}$, small indeed. 

Now, let $n = \infty$. Pick two Erd\H{o}s-Renyi $\frac{1}{2}$-graphs at random, independently (flip independent fair coins for each pair of vertices for each 
graph).  What is the chance they are isomorphic? The surprising answer: the chance is 1 (!). This accounts for the label ``the random graph.'' There are many non-random 
constructions. Let $S$ be the set of all primes that are $1 \mod 4$. Put an edge from $p$ to $q$ if $(\frac{p}{q}) = 1$. [The Legendre symbol $(\frac{p}{q})=1$ if $q$ is a 
quadratic residue mod $p$; by reciprocity, this happens if and only if $(\frac{q}{p}) = 1$ so the graph is undirected.] The resulting graph is (isomorphic to) the random graph. \footnote{\possible{Added when updating the manuscript: In response to the present paper, 
to further investigate the discontinuity in passing from Erd\H{o}s-Renyi graphs $G(n,\frac{1}{2})$ 
to the random graph $\mcr$, 
Chatterjee and Diaconis asked in \cite{c-j}: Pick independent $\Gamma_1, \Gamma_2$ $G(n, \frac{1}{2})$ graphs. What is the size of the largest 
induced isomorphic subgraph? They show it has order $4 \log_2 n$.} 

In a similar vein, Alon \cite{alon} had shown that a random graph in $G(N, \frac{1}{2})$ is universal in containing all subgraphs of order $k$, where 
$k$ is of order $\log (N)$ (he has more precise results).}

These graphs are robust. If $\mcr$ is the random graph and $\mcr^\prime$ is obtained from $\mcr$ by deleting any finite number of vertices and edges, then 
$\mcr^\prime \cong \mcr$. If the vertex set is decomposed into finitely many subsets, the induced graph on at least one subset is isomorphic to $\mcr$. 

The random (Rado) graph has a universality property; it contains every finite or countable graph as an induced subgraph. 
In particular, it contains a copy of the infinite 
complete graph and an infinite graph with no edges. The Rado graph is highly symmetric: given any two finite isomorphic induced subgraphs there is an automorphism of 
$\mcr$ extending this isomorphism. 

Logicians have long studied $\mcr$ -- after all, a graph is just a symmetric binary relation. Say that a graph property $\mbp$ holds in almost all finite graphs if the 
proportion of $N$-vertex graphs where $\mbp$ holds tends to $1$ (equivalently, if an Erd\H{o}s-Renyi $\frac{1}{2}$-graph has property $\mbp$ with probability tending 
to $1$). It is known that for any first order property in the language of graphs, $\mbp$ holds in $\mcr$ if $\mbp$ holds in almost all (finite) graphs. 
[This can be said as a zero-one law: let $\mbp$ be a sentence in the first-order language of graphs. Then $\mbp$ holds in almost all graphs or in almost no graphs.]  We hasten to add that many interesting 
properties of graphs are not first order, connectedness and hamiltonicity, for instance; and much random graph theory is done with the chance of an edge tending to 
zero as $n$ tends towards infinity. For much more on logic and graph theory, see \cite{spencer}.

Above, we explained how the study of a random walk and the associated Laplacian gives insight into the geometry of a space. 
There is a curious random walk on $\mcr$ leading to open math problems and new understanding of the structure of $\mcr$. 
To explain, it is easiest to use the following isomorphic description of $\mcr$: Let $\mathbb{N} = \{ 0, 1, 2, \dots \}$. Make a graph on 
$\mathbb{N}$ via: for $i<j$, set $i\sim j$ if and only if the $i$-th bit of $j$ (in its binary expansion) is a $1$. Thus, for instance, 
$0$ is connected to all odd numbers and $1$ is connected to $0$ and 
$1$ is connected to all $j$ equivalent to $2$ or $3$ $\mod 4$.  
For the walk, fix $Q(j) > 0$, $\sum_{j \in \mathbb{N}} Q(j) = 1$ a positive probability on $\mathbb{N}$. For $j \in \mathbb{N}$, let 
$N(j) = \{ k : k \sim j \}$ be the neighborhood of $j$. For definiteness, say $Q(j) = \frac{1}{2^{j+1}}$. (The same story works for 
any model of $\mcr$ and the rate of convergence would be the same; of course this depends on the measure.) 
Define a Markov chain on $\mathbb{N}$ by: 
\[ K(i,j) = 
\begin{cases}
Q(j)/Q(N(i)) & \text{if } j \sim i \\
0 & \text{otherwise}.
\end{cases}
\]
Thus, from $i$, pick $j$ among the neighbors with probability $Q(j)$ (normalized to the neighbors). 

We were surprised to find this walk has a simple stationary distribution $\Pi$ (so $\sum_{i \in \mathbb{N}} \Pi(i) K(i,j) = \Pi(j) )$. 
That is, 
\[  \Pi(i) = Z^{-1} Q(i)Q(N(i))         \]
for $Z$ a normalizing constant. Indeed $\Pi(i)K(i,j) = \Pi(j)K(j,i)$ as is easy to check, so 
\[ \sum_i \Pi(i)K(i,j) = \sum_i \Pi(j)K(j,i) = \Pi(j). \]
Standard theory says that, for any $i,j$, $K^\ell(i,j) \rightarrow \Pi(j)$ as $\ell \nearrow \infty$. (See \cite{levin}.)  Here  
$K^\ell(i,j) = \sum_k K(i,k) K^{\ell-1}(k,j)$ is the chance of going from $i$ to $j$ in $\ell$ steps. 
The question is to determine the \underline{rate} of this convergence. Let
\[   \| K^\ell_i - \Pi   \| = \frac{1}{2} \sum |K^\ell(i,j) - \Pi(j) |.    \]
We want to know how large to take $\ell$ so $\| K^\ell_i - \Pi \| < \epsilon$, and how it depends on the starting state $i$: 
each $i$ is connected to half of the points in $\mathbb{N}$ and the diameter of $\mcr$ is two.  So, convergence to stationarity might be very rapid 
(a bounded number of steps suffice to make $\| K^\ell_i - \Pi \| < \epsilon$ for all $i$). On the other hand, if $i$ begins 
with many zeroes, it can take a long time for the chain to get close to zero. We don't know and hope one of our readers 
will report. 


\sbr

\possible{
The following computation shows that the starting state matters. 

It is easiest to use the Boolean model for $\mcr$;  the vertex set is $\{ 0, 1, 2, \dots \}$ and for $i<j$ there is an undirected edge from 
$i$ to $j$ if the $i$th bit of $j$ is a $1$. Take the driving measure to be $Q(j) = \frac{1}{2^{j+1}}$. (As explained below, the choice of $Q$ 
doesn't matter much; essentially the same argument works for $Q(j) = \frac{1}{(j+1)(j+2)}$, $0 \leq j < \infty$.) 

The transition density is $K(i,j) = Q(j)/Q(N(i))$. The stationary distribution is $\Pi(i)= Q(i)Q(N(i))/Z$ with $Z$ a normalizing constant. 
Observe that $\Pi(0) = \frac{\frac{1}{2} \cdot \frac{1}{3}}{Z}$. Since $\Pi(0) < 1$, $1 < \frac{1}{Z} < 6$. 
Let 
\[ 2^{(k)} \mbox{ denote } 2^{2^{\cdot^{\cdot^{\cdot^2}}}} \]  
with the exponentiation iterated $k$ times. So $2^{(1)} = 2$, $2^{(2)} = 4$, $2^{(3)} = 16$ 
and we stipulate $2^{(0)} = 1$. Below, the easily proved inequality $2^{(k-1)} \leq 2^{(k)} - 2^{(k-1)} \leq 2^{(k)}$ will be used. 
The idea of the argument is simple. Start the Markov chain at $2^{(k)}$. The only smaller $j$ connected to $2^{(k)}$ is $2^{(k-1)}$. 
The larger $j$'s have super-exponentially small probability. Thus, in the first step, the walk goes $2^{(k)} \rightarrow 2^{(k-1)}$ with 
probability super-exponentially close to $1$. Similarly, it goes $\rightarrow 2^{(k-2)} \rightarrow 2^{(k-3)} \rightarrow \cdots \rightarrow 2^{(k-\ell)}$ 
in the first $\ell$ steps. For any fixed $\ell$, the walk started at $2^{(k)}$ is most likely \emph{not} at $0$. But 
$\Pi(0) \geq \frac{1}{6}$ so the walk cannot have converged after $\ell$ steps.

\begin{prop} \label{rg-prop}
With notation as above, 
\[ P_{2^{(k)}} (X_\ell = 2^{(k-1)}) \geq 1-\frac{1}{2 \cdot 2^{(k-\ell)}} ~~~. \]
\end{prop} 

\begin{cor} \label{rg-cor}
For $k > \ell \geq 2$, 
\[    \|  K^\ell_{2^{(k)}} - \Pi  \|  \geq \frac{1}{6} - \frac{1}{2 \cdot 2^{(k-\ell)}} ~~~.      \]
\end{cor}

\begin{proof}[Proof of Corollary \ref{rg-cor}]
By definition  $ \|  K^\ell_{2^{(k)}} - \Pi  \|  = \sup_A  \|  K^\ell(2^{(k)}, A) - \Pi(A)  \|$.   Take $A = \{ 0 \}$. Then 
\[  \|  K^\ell_{2^{(k)}} - \Pi  \| \geq \Pi(0) - K^\ell(2^{(k)}, 0) \geq \frac{1}{6} - \frac{1}{2 \cdot 2^{(k-\ell)}} ~~~.     \]
\end{proof}

\begin{proof}[Proof of Proposition \ref{rg-prop}] 
Let $N^+ (2^{(k)}) = \{ j > 2^{(k)} $ with $j \sim 2^{(k)} $ in $\mcr \}$.   So 
\[ Q(N^+(2^{(k)})) \leq \sum_{j > 2^{(k)}} Q(j) = Q(2^{(k)}) = \frac{2}{2^{(k+1)}}. \]
Also, 
\[ K(2^{(k)}, 2^{(k-1)}) = \frac{ 1/ (2 \cdot 2^{(k)} ) } { 1/(2 \cdot 2^{(k)})  + Q(N^+(2^{(k)}))}  \geq \frac{1}{1+\frac{2^{(k)}}{2^{(k+1)}}} \geq 1-\frac{ 2^{(k)} }{ 2^{(k+1)} }. \]
(Use $\frac{1}{1+x} > 1-x$ for $x >0$.) Using $(1-x)(1-y) \geq 1-x-y$, 
\begin{align*}
P_{2^{(k)}}\left( X_\ell = 2^{(k-\ell)} \right) \geq & ~ 1 - \frac{ 2^{(k)} }{ 2^{(k+1)} } - \frac{ 2^{(k-1)} }{ 2^{(k)} }
 - \cdots - \frac{ 2^{(k-\ell)} }{ 2^{(k - \ell+1)} } \\
 \geq & ~ 1 - \frac{1}{2^{(k)}} - \cdots - \frac{1}{2^{(k-\ell)}} \geq \frac{1}{2 \cdot 2^{(k - \ell)} }
\end{align*}
which completes the proof.
\end{proof}

\begin{rmk-star} 
\emph{The arguments above are robust. They work if the starting state $2^{(k)}$ is replaced by any $j$ with first non-zero bit in position $2^{(k-1)}$. 
They work if the measure $Q$ is replaced by any monotone decreasing probability, e.g. $1/(j+1)(j+2)$.}  
\end{rmk-star}
}


\subsection{Quasirandomness}

We now turn to \emph{finite} random graphs.   Although, as we have just 
observed, there are many nonisomorphic Erd\H{o}s-Renyi $\frac{1}{2}$-graphs of a fixed finite size $N$, they generally share a series of rather special properties.  
For example, almost always (as $N$ grows),  
the edge distribution between pairs of reasonably sized subsets of 
vertices is regular in the sense of Szemer\'edi's regularity lemma; most vertices have degree about $\frac{n}{2}$; 
all labeled graphs on a fixed finite number of vertices occur asymptotically with the same frequency -- 
say, (labeled) $4$-cycles are not less frequent than cliques of size $4$ or independent sets of size $4$.

One of the remarkable discoveries of \cite{c-g-w}, \cite{thomason} was that it is possible to identify a set of such (asymptotic) 
graph properties, a priori quite different from each other but shared by 
random graphs, which are in fact equivalent to each other 
in the sense that any graph satisfying one of these properties must necessarily satisfy all of them. Such graphs are called \emph{quasi-random}. 
Quasi-random graphs form a broader class than random graphs, but nonetheless this framework gives a powerful and often easily checkable means of speaking approximately or 
asymptotically of random behavior. 

There is by now an extensive literature on this subject. 
Some variously formulated lists of the equivalent properties can be found in \cite{c-g-w} \S 3 (note that for ease of exposition they restrict there to edge probability $\frac{1}{2}$) or in \cite{sudakov} Theorem 2.6 or \cite{gowers} Theorem 2.1.  For a textbook development see \cite{lovasz}.

\sbr

\subsection{First-order axioms}  
When considering the infinite Heisenberg groups, of which there are many of arbitrarily large infinite size, 
we switch from writing $H_\infty$ to $H_\omega$ to indicate \emph{the} one which is countably infinite. Here we indicate why this is justified for countable, though as already noted above,  
not uncountable sizes, by way of reviewing axioms and notation. 

Recall that for $p$ an odd prime, 
\cite{felgner} gave a set of first-order axioms $T_p$ in the language of groups (there is a binary function symbol $\times$ and 
a constant $1$) which hold in $H_\omega(p)$ and 
which express: the axioms for group theory (associativity: for all $x$, $y$, $z$ we have $x\times(y \times z) = (x \times y) \times z$, 
identity: for all $x$, $x \times 1 = 1 \times x = x$, existence of inverse: for all $x$ there exists $y$ such that $x \times y = y \times x = 1$),  that the center is a cyclic group of order $p$, that the derived group is contained in the center and is not trivial, and 
(infinitely many axioms expressing that) the factor group modulo the center is infinite.  Note this avoids direct reference to the Frattini subgroup, the intersection of all 
maximal subgroups, which is not obviously 
first-order expressible since it quantifies over subgroups, rather than elements. 

Logicians have various notions of equivalence that are subtly different.  For a model $M$, 
write $M \models T$ to mean that the set of axioms $T$ all hold in $M$. 
($T$ for {theory}, which is just a set of axioms, here always of first-order logic.) 
Write $M \equiv N$, pronounced elementarily equivalent, to mean that exactly the same axioms hold in $M$ and in $N$. Write $M \cong N$ 
to mean they are isomorphic. $M \cong N$ implies $M \equiv N$, but the reverse fails strongly by the upward and downward L\"owenheim-Skolem 
theorems: if $M$ is an infinite model, say in a countable\footnote{If larger, we may only have models of any size greater or equal to the size of the language.} language, then there is at least one $N \equiv M$ of 
every other infinite size. (See \cite{c-k} for standard model theory.) 
So to have a chance of determining the structure up to isomorphism we must specify not only the theory but the size, but often this is not enough: 
for example, there are many pairwise 
non-isomorphic algebraically closed fields of characteristic zero, and countable size: one of transcendence degree $n$ for 
every $n \in \mathbb{N}$, and one of countable transcendence degree. (However, there is just one of any fixed uncountable size, because 
in that case the transcendence degree and thus the isomorphism type is determined.)  When it is enough -- when there is, up to isomorphism, 
precisely one way to satisfy the axioms of $T$ on a set of size $\kappa$ --  $T$  is called $\kappa$-categorical. 

Summarizing ``there is exactly one countably infinite extraspecial $p$-group of exponent $p$'' 
in our notation: Felgner proved that $T_p$ is $\aleph_0$-categorical, 
meaning that if $M \equiv H_\omega$ (or just: $M \models T_p$) and $M$ 
is also countably infinite, then $M \cong H_\omega$. (The background statement is that if $p$ is an odd prime, $H$ is a nonabelian finite or countable 
$p$-group of exponent $p$ such that $H^\prime = Z(H)$ is cyclic and $H/Z(H)$ is elementary abelian, then whenever $D_1, D_2$ are finite 
subgroups of $H$ which are both extraspecial and both of the same size, every isomorphism from $D_1$ onto $D_2$ can be extended to an 
automorphism of $H$.) 
It follows that the theory $T_p$ is complete, meaning that for any other first-order sentence $\vp$ in its language, 
either $\vp$ or $\neg \vp$ follows from $T_p$.

A note on first-order logic. Part of the art (in mathematics) is of course finding a balance between having 
enough expressive power to prove interesting theorems but not so much as to prevent abstraction.   First order logic, in this sense, seems not unlike  
linear algebra: much interesting mathematics initially appears to escape it, but as, say representation theory shows for linear algebra 
(or, say, classification theory for first-order logic), already there is a remarkable explanatory power.

\sbr
\possible{
\subsection{Symplectic spaces} \label{s:sym} 
There is a helpful correspondence between our extraspecial $p$-groups and non-degenerate symplectic spaces over $\mathbb{F}_p$, 
which we outline here with the referee's encouragement (and will refer to in section four). 
For more details see the excellent account in \cite{tomkinson}, pages 49--53.

In the first direction, we are given an extraspecial $p$-group $H$ of exponent $p$ and we would like to find a 
nondegerate symplectic space over $\mathbb{F}_p$. 
Recall from above that the commutator is a map from 
$H \times H \rightarrow Z(H)$ and only depends on the conjugacy classes, i.e., on the cosets of the center: 
the commutator of 
$[x,y,z]$ and $[a,b,c]$ is $[0,0, xb-ay]$. Recall also that the quotient group mod the center $G = H/Z(H)$ is elementary abelian so, written 
additively, can be thought of as a vector space $V$ over $\mathbb{F}_p$. [For the rest of this paragraph, write elements of $V$ as $(u,v)$, i.e. 
$(u,v) \in V$ is the image of $[u,v,*] \in H$ in $H/Z(H)$, now using round parentheses  
to not conflict with the notation for commutator.] 
Let $g = [0, 0, 1] \in Z(H)$ (or any other generator of the center). Define   
$F : V \times V \rightarrow \mathbb{F}_p$ by sending $( (u_1, u_2), (v_1, v_2)) \mapsto k$ for the unique $0 \leq k \leq p-1$ such that 
the commutator  $[ [u_1, u_2, *], [v_1, v_2, *] ] = [0, 0, k] = g^k$. 
Then $F$ is (i) bilinear, (ii) \emph{alternating} meaning that $F( (x_1, x_2), (x_1, x_2) ) = 0$ (recall each $[x_1, x_2, *]$ commutes with 
all elements of its conjugacy class), 
and (iii) \emph{nondegenerate} meaning that for every nonzero $(u_1, u_2) \in V$ there is $(v_1, v_2) \in V$ 
such that $F( (u_1, u_2), (v_1, v_2) ) \neq 0$.  

Compare our $\tilde{\Gamma}(H)$, whose edges and non-edges reflect whether $F$ is zero. 

In the second direction, we have a nondegenerate symplectic space $V$ over $\mathbb{F}_p$, given with a basis 
$\{ v_i : i \in I \}$ and an alternating bilinear map $F$, and we would like to 
find an extraspecial $p$-group of exponent $p$. Informally, we try to reverse the earlier quotient by remembering a center. 
Let $H$ be the group formally given by generators $\{ g \} \cup \{ v_i : i \in I \}$ subject to the relations $g^p = {v_i}^p = [ v_i, g] = 1$ and 
$[v_i, v_j] = g^{F(v_i, v_j)}$. Most properties are easily checked; the proof that these relations don't collapse goes   
by finding an explicit construction of a nonabelian group which satisfies them 
(see e.g. \cite{tomkinson} p. 52, or \cite{sh-st}).

In this sense our Heisenberg groups and symplectic spaces $(\operatorname{mod} p)$ are equivalent.
}

\sbr

\subsection{Commuting graphs} \label{s:c-g} 
The commuting graph of a finite group, our $\Gamma$ of the introduction, also has an extensive literature.  
\cite{pyber} includes many early references; a very recent survey of many kinds of graphs on groups, 
including the commuting graph, is \cite{cameron-2}. 
There is a direct connection to randomness in the question: Let $G$ be a finite group. Pick two elements in $G$ at random, what is the chance 
they commute? An early theorem in the subject shows that, for $G$ non-abelian, this is at most $5/8$ (and this is sharp). For simple groups 
 \cite{gr} show that this chance is at most $1/2$. For non-abelian solvable groups 
the chance is at most $1/12$. This last paper surveys a surprisingly deep literature. 

We continue this discussion with two comments.  
First,  important early questions in this area are suggested by the title of \cite{e-s}: ``How abelian is a finite group?'' and the 
above-mentioned companion paper 
\cite{pyber} of the same name. 
The maximal cliques of the commuting graph $\Gamma$ correspond to maximal abelian subgroups, however, 
the picture is a priori different from Ramsey's theorem: 
Pyber proves that every group of
order $n$ contains an abelian subgroup of order at least $2^{\epsilon \sqrt{\log n}}$ for some $\epsilon > 0$ and that this  
result is essentially best possible.  

Second, we include the short: 
\begin{proof}[Proof of Fact \ref{et-fact} \emph{(Erd\H{o}s-Tur\'an)}] 
The number of ordered pairs of commuting elements whose first element is $a$ is $|Z(a)|$, the size of the 
centralizer of $a$. This number is the same for any $a^\prime \in C(a)$, the conjugacy class of 
$a$. So the number of ordered pairs of commuting elements with first element conjugate to $a$ is $|C(a)| |Z(a)| = |G|$. 
Sum over $c(G)$ conjugacy classes to get $|G| c(G)$ ordered pairs. 
\end{proof}

\vspace{8mm}



\begin{center}

\includegraphics[width=100mm]{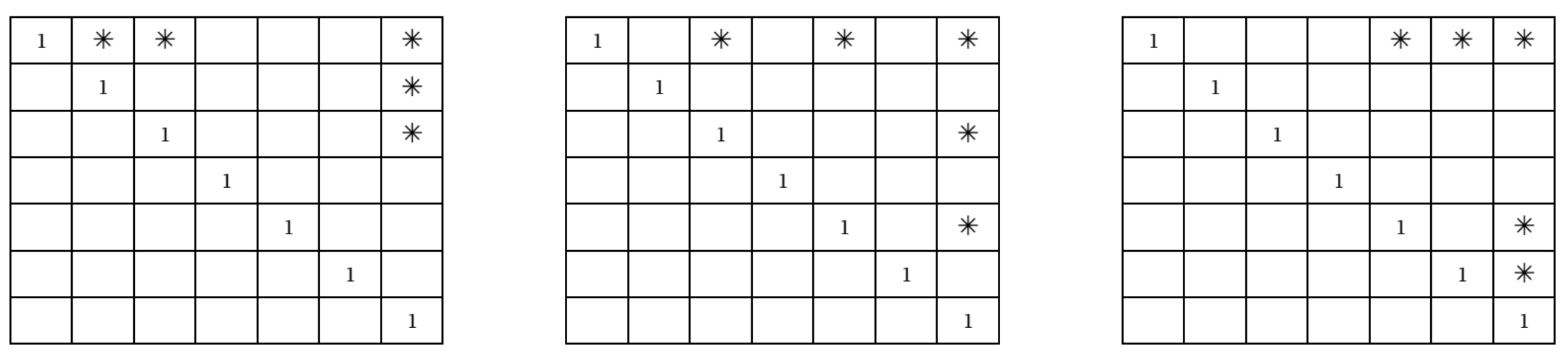}

\textcolor{gray}{\small{\textsc{Figure 1.} Some copies of $H_{2m+1}$ in $H_{2n+1}$ for $n=5, m=2$.}}

\end{center}

\newpage


\section{Quasirandomness}

This section gives a standalone proof that the sequence of commuting graphs of the Heisenberg group is quasi-random. Some of the 
properties that follow have other explanations, see for instance section 4 below. 

Recall the graph $G = \tilde{\Gamma}(H_{2k+1})$ discussed above has vertices the non-central conjugacy classes 
$C_{x,y}$ with an edge from $C_{x,y}$ to $C_{a,b}$ if the elements of these two classes commute. Combining $x$ and $y$ into a 
vector of length $2k$ it is convenient to identify vertices with functions from $\{ 1, \dots, 2k \}$ to $p$, 
including the all-zero function 
so $n = |V(G)| = p^{2k}$. For any $v \in V(G)$ and $1 \leq i \leq 2k$, we may write $v(i)$ for the $i$-th entry of $v$. 
Let the support of $v$, $\supp(v) = \{ i : 1 \leq i \leq 2k, v(i) \neq 0 \}$. Let $N(v)$ denote the neighborhood of $v$.  
We emphasize a notational point: in this section, $k$ is from the subscript to $H$ and $n$ is the size of the graph. 

\begin{claim} \label{c:count1} Given $k$ and $p$, we have that: 
\begin{enumerate}
\item Every $v \in V(G)$ other than the constant-zero element has degree $d = p^{2k-1}$.
\item The density $\delta$ of $G$ is $\approx \frac{1}{p}$, more precisely, $\delta = \frac{1}{p} + \epsilon$ where $\epsilon = \epsilon(k) < \frac{1}{p^{2k}}$. 
\item For $v,w \in G$, we can have:
\begin{enumerate}
\item $|N(v) \cap N(w)| = p^{2k}$, if $v=w$ is the constant-zero element. 
\item  $|N(v) \cap N(w)| = p^{2k-1}$, if $v=w$ is not the constant-zero element, or if exactly 
one of $v,w$ is the constant-zero element, or if $v$ is nonzero and $w$ is a nonzero multiple of it. 
\item $|N(v) \cap N(w)| = p^{2k-2}$ otherwise. 
\end{enumerate}
\end{enumerate}
\end{claim}

\begin{proof}

(1) For any $v \in G$ which is not constantly zero, we can count the elements $w$ which commute with $v$ by fixing some $i \in \supp(v) \neq \emptyset$ and choosing 
$w(j)$ freely for $j \neq i$.   Thus $\deg(v) = p^{2k-1}$. 

(2) Counting ordered pairs, compute the density of $G$ by $( (n-1)p^{2k-1} + p^{2k} ) / n^2 = \frac{1}{p} + \frac{p-1}{p^{2k+1}} < \frac{1}{p} + \frac{1}{p^{2k}}$, remembering $n = p^{2k}$. 

(3) Case (a) is clear. Otherwise, one of the following must be true:  
\begin{itemize}
\item[(i)] we can choose $1 \leq i, j \leq 2k$ such that $i \in \supp(v)$, $j \in \supp(w) \setminus \supp(v)$. In this case $|N(v) \cap N(w)| = p^{2k-2}$.
[Looking for $x$ commuting with both, we may choose $x(\ell)$ for $\ell \neq i,j$ freely, but the value of $x(i)$ is forced by commuting with $v$ and subsequently the 
value of $x(j)$ is forced by commuting with $w$.]

\item[(ii)] the parallel to (i): we can choose $1 \leq i,j \leq 2k$ such that $i \in \supp(w)$, $j \in \supp(v) \setminus \sup(w)$, and the proof is the same.   

\item[(iii)] not (i) or (ii), so $\supp(v) = \supp(w)$, but one is not a multiple of the other.  In this case, to find $x$ commuting with both, we are solving two equations in $|\supp(u)| \geq 2$ unknowns, so we are free to choose $p^{2k-2}$ coordinates of $x$ as we like. 

\item[(iv)] $v$ and $w$ are nonzero multiples of each other; 
in this case, $N(v) \cap N(w) = N(v) = N(w)$ and there are $p^{2k-1}$ such $x$. 
Note that there are $(n-1)(p-2)$ 
 such pairs $(v,w)$ with $v \neq w$. 
\end{itemize}

This completes the proof. 
\end{proof}

Our graph is thus close to being regular: all vertices but one have the same degree $p^{2k-1}$ (and the remaining one has degree $p^{2k}$). 
Recall that a graph is called strongly regular if there are $A, B$ such that if $v, w$ have an edge between them then $|N(v) \cap N(w)| = A$, and 
if $v, w$ do not have an edge between them then $|N(v) \cap N(w)| = B$. 
Our $G$ is in in some sense also close to being strongly regular.

\begin{theorem} \label{t:psr}
The sequence of graphs 
\[ \langle \tilde{\Gamma}(H_{2k+1}) : k \rightarrow \infty \rangle \] 
is quasirandom. 
\end{theorem}

\begin{proof}
It will suffice, see e.g. the formulation of \cite{sudakov} Theorem 2.6 item P7, to  
verify 
that for our density $\delta = \frac{1}{p} + \epsilon$,
\begin{equation} \label{e:sum}
\sum_{v, w \in G} | |N(v) \cap N(w)| - \delta^2 n| = o(n^3). 
\end{equation}
By Claim \ref{c:count1} this sum has several kinds of components. For each of the pairs $(v,w)$ in case \ref{c:count1}(3)(c), we have 
$| |N(v) \cap N(w) | - \delta^2 n| = | p^{2k-2} - (\frac{1}{p} + \epsilon)^2 n| = | n/p^2 - n/p^2 - \epsilon^2 n - 2\epsilon/n | 
< | \frac{1}{n} + \frac{2}{p}|$, recalling that $\epsilon < \frac{1}{p^{2k}} = \frac{1}{n}$. There are fewer than $n^2$ such pairs, so all together 
$n^2 (\frac{1}{n} + \frac{2}{p}) = o(n^3)$.  Otherwise, we have one $(v, w)$ in case $\ref{c:count1}(3)(a)$, and 
$(n-1) + (n-1) + (n-1)(p-2) 
 \leq p^2 n$ pairs in $\ref{c:count1}(3)(b)$.
So the left side of equation (\ref{e:sum}) for these pairs is bounded above by 
$| p^{2k} - \delta^2 n| + p^2 n |p^{2k-1} - \delta^2 n| = |n - \delta^2 n| + p^2 n |\frac{n}{p} - \delta^2 n| \leq n(1-\delta) + n^2 ( p + p^2 \delta^2) = o(n^3)$. 
\end{proof}

Thus we have all the equivalent formulations of quasirandomness, for example: 

\begin{cor} \label{c:52}
For all $A, B \subseteq V(G)$, we have that $e(A, B) = \frac{1}{p}|A||B| + o(n^2)$. 
\end{cor}

Let us verify that the analogous result transfers from the quotient $\tilde{\Gamma}$ to $\Gamma$. Recall that in $\Gamma$, 
each point of $\tilde{\Gamma}$ blows up to a clique on $p$ vertices (grouping together the $p$ central classes in our picture to form the blow-up of $0$). 
Write $[v]$ for the set of vertices in the blow-up of $v \in \tilde{\Gamma}$. Then if $a \in [v], b \in [w]$ we have that 
$(a,b)$ is an edge in $\Gamma$ if and only if $(v,w)$ is an edge in $\tilde{\Gamma}$.  So the graph $\Gamma$ has 
$N = np = p^{2k+1}$ vertices and each vertex $x \in \Gamma$, $x \notin [0]$ has degree $p \cdot p^{2k-1} = p^{2k}$.
Recalling \ref{et-fact}, we may compute density via ordered pairs of edges in $\Gamma$ by 
$N(p^{2k}+p-1)/N^2 = (p^{2k}+p-1)/p^{2k+1} = \frac{1}{p} + \epsilon$, where $\epsilon = (p-1)/p^{2k+1} < \frac{p}{N}$ goes to $0$ as $k \rightarrow \infty$. 
The exponents in Claim \ref{c:count1}(3) go up by one in each case and so the parallel count in Theorem \ref{t:psr} goes through 
(with $N$ instead of $n$). 

\begin{concl} \label{c:57}
The sequence of graphs 
\[ \langle {\Gamma}(H_{2k+1}) : k \rightarrow \infty \rangle \] 
is quasirandom.  
\end{concl}

\begin{conv} 
For any group $G$ and $A \subseteq G$ a subset or subgroup, write $\Gamma(A)$ to mean the induced subgraph of $\Gamma(G)$ formed by restricting 
the vertex set to $A$. Likewise, for $A, B \subseteq G$, write $\Gamma(A \times B)$ for the corresponding bipartite graph, where bipartite in this case 
\emph{means} that we ignore the edges on $A$ and on $B$, not that we require them not to exist. 
\end{conv}

\begin{disc}
We may also conclude from this that the sequence of \emph{bipartite} graphs $\Gamma(H_{2m+1} \times H_{2m+1})$ as $m \rightarrow \infty$ 
is quasirandom.  
Since we had allowed self-loops in $\Gamma$ and $\tilde{\Gamma}$, 
the calculations are the same. 
\end{disc}

\br

\section{Rado-ness}

This section works directly with infinite extraspecial $p$-groups $H_\omega(p)$, for $p$ an odd prime. 
The main result shows that the commuting graph $\tilde{\Gamma}(H_\omega)$ contains 
an induced copy of the Rado graph in quite a strong way: the vertices of $\mcr$ (together with the center $Z(H_\omega)$) generate $H_\omega$.  
\possible{We include different proofs of the various lemmas (one due to the referee) with different advantages.}
At the end, we discuss what these results say for finite Heisenberg groups.

Observe that $H_{2n+1}$ appears as a subgroup of $H_{2m+1}$ for $m\geq n$, and also as a subgroup of $H_\omega$ 
(consider those $[x,y,z]$ in the larger group whose $x$ and $y$ have nonzero entries only in the 
first $n$ places; it's sufficient to show \emph{this} $H_{2n+1}$ appears, recalling the finite extraspecial $p$-groups of 
exponent $p$ are determined up to isomorphism by their size). So it follows from quasirandomness\footnote{\label{footnote-4}
This is easier than 
quasi-randomness and can be seen directly.  Consider $n=3$. Choose as a first vertex 
$(1,0,0), (0,0,0)$. Choose as the second vertex $(0,1,0), (a,0,0)$ where $a$ is $0$ or $1$ to control commuting with the first vertex. 
Choose as the third vertex $(0, 0, 1), (b,c,0)$ where $b, c$ are chosen to be $0$ or $1$ to control commuting with the second and third 
vertices respectively.}  
 that:\footnote{ For this section the 
sizes involved won't matter, however the previous footnote suggests $H_{2n+1}$ is universal for graphs of size $n$.} 

\begin{cor} \label{c:41} 
For every finite $k$ there is $n_*=n_*(k)$ such that every graph on $k$ vertices appears as an induced subgraph of 
$\tilde{\Gamma}(H_{2n+1})$ for every $n \geq n_*$. 
\end{cor}

\begin{concl} \label{c:42}
In $\tilde{\Gamma}(H_\omega)$, every finite graph appears as an induced subgraph.  $($It follows by choosing representatives 
from each of the conjugacy classes involved that the same is true in $\Gamma(H_\omega)$.$)$ 
\end{concl}

To motivate what follows: with a little more work (explained below),
 we can find the full Rado graph $\mcr$ inside $\tilde{\Gamma}(H_\omega)$ as an induced subgraph, 
but we might wonder how integral this subgraph is to the structure of the whole graph. 
 Since $\tilde{\Gamma}$ is defined in terms of conjugacy classes, a notion of a group ``built over a graph'' is introduced as a surrogate for 
 ``the vertices of $\mcr$ generate $H_\omega$.''  
  
\begin{defn}
Say that the extraspecial $p$-group $H$ is \emph{built over} the graph $R$ if there is $X \subseteq \Gamma(H)$ 
such that:
\begin{enumerate}
\item the set $X$ is a union of conjugacy classes 
\item $\tlg(X)$ and $R$ are isomorphic as graphs
\item $H$ is equal to the subgroup of $H$ generated by $X \cup Z(H)$. 
\end{enumerate}
\end{defn}


\begin{claim}  \label{c:42}
Suppose $H$ is an infinite extraspecial $p$-group and $X \subseteq H$ is infinite and \possible{for every $x \in X$ there 
is $y \in X$ which does not commute with $x$}. Let 
$G$ be the subgroup of $H$ generated by $X \cup Z(H)$.  Then $G$ is an extraspecial $p$-group also.  
Note that $|G| = |X|$. 
\end{claim}

\possible{
\begin{rmk}
Why the hypothesis on $X$? 
In the language of $\S \ref{s:sym}$, 
if $V$ is an infinite nondegenerate symplectic space and we are looking for subspaces corresponding to subgroups which are also extraspecial, 
they should be nondegenerate. Note that if  
$X$ is an infinite set of elements of $V$ which is nondegenerate in the sense that for every $x_i \in X$ there is $x_j \in X$ 
which does not commute with it, then the subspace $W$ generated by $X$ is also an infinite nondegenerate symplectic space. 
\end{rmk}
}

\begin{proof}[Proof of Claim $\ref{c:42}$]
Let us check that it satisfies Felgner's axioms $T_p$ \cite[p. 423]{felgner}.  
By assumption, $G$ satisfies the axioms for group theory. 
$C_p = Z(H) \subseteq G$. 
Since $G$ is a subgroup of an extraspecial group and is not abelian, it must be normal. So $Z(H) \subseteq Z(G) = \{  a \in G : $ the conjugacy class of $a$ has size $1$ in $G$ $\}$. 
\possible{By our assumption on $X$, 
any other element of $G$ whose conjugacy class has size $1$ in $G$ must still have this property in $H$ and so 
be in the center of $H$. }
So indeed (ii) the center is a cyclic group of order $p$.  Axiom (iii), which says that the derived group is contained in the center, 
corresponds to a universal statement so automatically passes to substructures.\footnote{The axiom as written says 
$\forall x \forall y \forall z (z x^{-1}y^{-1}xy = x^{-1}y^{-1}x y z)$, which a priori hides quantifiers in the expressions ``$x^{-1}$'' and 
``$y^{-1}$'' since our language has just $\times$ and $1$. However, for any elements $a, b, c, d$ in $G$ such that 
$ab=ba=1$ and $cd=dc=1$, we have that $\vp[a,b,c,d] = \forall z (z bdac= bdac z)$ holds in $H$ thus in $G$; or just 
observe that $\forall x \forall y \forall v \forall w \forall z ( (x v = v x =1 \land y w = w y = 1) \implies 
(z w v x y = w v x y z))$ is truly universal, and note $G$ contains all necessary inverses. }
Axiom (iv), which says the derived group is not trivial, is satisfied 
since $X$ is assumed to contain two elements which do not commute. Axiom (v) says every element has order $p$, which (is universal and) 
remains true. 
Axiom (vi) says the factor group modulo the center is infinite, 
which is true because $X$ is infinite and the center has size $p$. 

For the last line, note that every element of $G$ can be written as a finite string in elements of $X \cup Z(H)$ and the operations times and 
inverse, and if $\kappa$ is infinite, then $\kappa^{<\omega}  = \kappa$, i.e. the set of finite sequences of elements of $\kappa$ 
has size $\kappa$. 
\end{proof}

\possible{We now look for infinite graphs $R$ in $\tilde{\Gamma}(H)$ in Lemma \ref{finding-r}.  
Note that in the case where $R$ and $H$ are both countable, which 
is really all we need for $\mcr$,  we may see this simply in several ways. (The reader who believes it 
may look ahead to \ref{c:induced}--\ref{t:rado}.)  

\begin{proof}[First proof: $\tilde{\Gamma}(H_\omega)$ contains any countable graph as an induced subgraph.] 
By Corollary \ref{c:41} and the compactness theorem of first-order logic, there is a countable extraspecial $p$-group $H \equiv H_\omega$ such that 
$\tilde{\Gamma}(H)$ contains a given countable graph $R$ as an induced subgraph. By $\aleph_0$-categoricity, $H \cong H_\omega$. 
\end{proof}

Thanks to the referee for suggesting the following proof.  
For the quoted result about symplectic spaces (whose proof uses countability) 
see for example \cite{tomkinson} Theorem 3.9. 

\begin{proof}[The referee's proof]
It is known that a countable space with a symplectic form $B$ has a basis 
$\{ e_1, e_2$, $\dots$, $f_1, f_2, \dots \}$ such that $B(e_i, f_j) = \delta_{ij}$, and $B(e_i, e_j) = B(f_i, f_j) = 0$.  
Let $v_1, v_2, \dots$ be the vertices of the Rado graph, or another given countable graph. 
Map $v_1$ to $e_1$. Now assuming $v_1, \dots, v_n$ have been 
mapped correctly, map $v_{n+1}$ to $e_{n+1} + \sum^n_{i=1} x_i f_i$, where $x_i = 0$ if $v_{n+1} \sim v_i$, and is $1$ otherwise. 
Now it is readily checked that the images of $v_i$ and $v_{n+1}$ have product $0$ with respect to $B$ if and only if 
$v_i$ and $v_{n+1}$ are adjacent.
\end{proof}
}

\possible{
\begin{disc}
\emph{
There is also a nice connection between this proof and footnote 4 (for $n = \omega$), 
explained by the fact (see e.g. \cite{tomkinson} Corollary 3.10) 
that when $p > 2$, a countably infinite extraspecial $p$-group of exponent $p$ is the direct product with amalgamated center of groups 
each isomorphic to $H_3(p)$.  In this sense, we can see $e_i$ as $[x, 0, 0]$ where $x$ has $1$ in the $i$-th place and zeroes elsewhere, and 
$f_i$ as $[0, y, 0]$ where $y$ has $1$ in the $i$-th place and zeroes elsewhere. 
} 
\end{disc}
}

\begin{lemma} \label{finding-r}
For any infinite graph $R$, there is an infinite extraspecial $p$-group $H$ [i.e., a model of $T_p$] whose $\tilde{\Gamma}(H)$ contains 
$R$ as an induced subgraph. 
\end{lemma}

\begin{proof}
We follow the classical proof of the compactness theorem via ultraproducts. 
Enumerate the set of vertices of $R$ as $\{ v_\alpha : \alpha < \kappa \}$ and enumerate the finite subsets of $\kappa$ as 
$\langle u_i : i \in I = [\kappa]^{<\aleph_0} \rangle$.  
For each $\alpha < \kappa$ let $I_\alpha= \{ i \in I : \alpha \in u_i \}$. Then 
the set $\{ I_\alpha : \alpha < \kappa \}$ has the finite intersection property, so may be extended to an ultrafilter 
$\de$ on $I$. 

For each $i \in I$, let $R_i$ be the finite induced subgraph of $R$ with vertex set $\{ v_\alpha : \alpha \in u_i \}$. Choose an isomorphic 
copy of this graph in $\tilde{\Gamma}(H_\omega)$ and choose a representative of each conjugacy class involved as a $\tilde{\Gamma}$-vertex; this 
amounts to choosing a set $A_i \subseteq H_\omega$ of elements in distinct conjugacy classes and a bijection $\pi_i : \{ v_\alpha : \alpha \in u_i \} \rightarrow A_i$ so that 
for any $\alpha, \beta \in u_i$, there is an edge between $v_\alpha, v_\beta$ in $R$ if and only if $\pi_i(v_\alpha)$ and $\pi_i(v_\beta)$ commute 
[i.e. if and only if there is an edge between $\pi_i(v_\alpha)$ and $\pi_i(v_\beta)$ in $\Gamma$].  

Work in the ultrapower $H = (H_\omega)^I/\de$, which is also an extraspecial $p$-group by \L o\'s' theorem 
(which implies the theory is preserved under ultrapowers: see \cite{c-k} Corollary 4.10). For each $\alpha < \kappa$ and each 
$i \in I$, define $\mathbf{v}_\alpha(i)$ 
to be $\pi_i(v_\alpha)$ if $\alpha \in u_i$, and any element of $H_\omega$ otherwise. For each $\alpha < \kappa$, define $\mathbf{v}_\alpha = \langle \mathbf{v}_\alpha(i) : i \in I \rangle/\de \in H$. Then the induced subgraph of $\Gamma(H)$ with vertex set $\{ \mathbf{v}_\alpha : \alpha < \kappa \}$ is isomorphic to $R$ via the map $\mathbf{v}_\alpha \mapsto v_\alpha$, since for any two $\alpha, \beta < \kappa$ the set 
$\{ i : \alpha, \beta \in u_i \} \in \de$, so there is an edge between $\mathbf{v}_\alpha, \mathbf{v}_\beta$ in $\Gamma(H)$ 
if and only if there is an edge between $v_\alpha$ and $v_\beta$ in $R$.  
Moreover, in $H$, $\{ \mathbf{v}_\alpha : \alpha < \kappa \}$ is a set of distinct elements, indeed a set of 
elements in distinct conjugacy classes, so there is a corresponding copy of $R$ in $\tilde{\Gamma}(H)$ 
replacing each $\mathbf{v}_\alpha$ by its conjugacy class. 
\end{proof}

\begin{cor} \label{c:induced}
There is an induced copy of the Rado graph $\mcr$ in $\tilde{\Gamma}(H_\omega)$. 
\end{cor}

\begin{proof}
Let $H$ be the extraspecial $p$-group given by Lemma \ref{finding-r} in the case where $R$ is the countable Rado graph $\mcr$ 
\possible{(which clearly satisfies the nondegeneracy hypothesis)}. Let 
$G$ be the subgroup of $H$ generated by $Z(H)$ and the (conjugacy classes of) elements forming the vertices of the copy of $\mcr$.  Then 
$G$ is countable and by Claim \ref{c:42} it is an extraspecial $p$-group, so by $\aleph_0$-categoricity it is isomorphic to $H_\omega$. 
\end{proof}

The proofs just given show:  

\begin{theorem} \label{t:rado}
$H_\omega$ can be built over the Rado graph $\mcr$.
\end{theorem}

In fact, the proof shows that $H_\omega$ can be built over any countable graph $G$ which is 
\possible{nondegenerate in the sense that no vertex has full degree in $G$}. 

\begin{disc}
\emph{Just for fun we briefly sketch a complementary proof of \ref{c:41} from the uncountable. 
\cite{sh-st} explicitly construct uncountable extraspecial $p$-groups (all of whose maximal abelian subgroups are small).  
By looking carefully at their construction it is possible to see directly 
how to construct arbitrarily long finite sequences of elements 
whose patterns of commuting and non-commuting 
can be freely chosen along the way.  
Since $T_p$ is complete and the statement that there exist $n$ elements no two of which are conjugate and which have 
a given pattern of commuting and non-commuting is first-order expressible (for any fixed finite $n$), 
it follows that $\tilde{\Gamma}(H_\omega)$ 
contains any finite graph as an induced subgraph.  It remains to derive \ref{c:41}.  
Fix a finite graph $G$. 
Let $V$ be a set of elements of $H_\omega$ which form the vertices of an induced copy of $G$.  
Let $X$ be the smallest subset of $H_\omega$ which is a union of conjugacy classes and contains $V$.  Choose $n$ minimal so that 
$H_{2n+1}$ contains $X$ (since $X$ is a finite set, it suffices to choose the minimal $m$ such that all the nonzero entries in 
$x,y$ of each $[x,y,z] \in X$  are among the first $m$ elements). Then for every $\ell \geq n$, $G$ appears as an induced subgraph of 
$\tilde{\Gamma}(H_{2\ell+1})$, and so (choosing representatives of the conjugacy classes) 
also of $\Gamma(H_{2\ell+1})$.  
}
\end{disc}

We may summarize sections three and four by saying that both the $H_{2n+1}(p)$'s and $H_\omega(p)$ 
may be reasonably (and in the finite case, quantitatively) understood as ``random'' objects.

\br
\br

\section{Towards a picture of $U_n$}



In this section we revisit a certain notoriously complicated object, $\UT(n,p)$, the group of $n\times n$ uni-upper-triangular matrices with entries 
in $\mathbb{F}_{p}$ (which we will soon abbreviate ``$U_n$''). 
This basic group arises as the Sylow $p$-subgroup of $\GL(n,p)$.   It is a ``universal $p$-group'' in the sense that every $p$-group is a subgroup of 
some $\UT(n,p)$. The center consists of all matrices in $\UT(n,p)$ which are zero except in the $(1,n)$-entry. The commutator subgroup equals the 
Frattini subgroup. These consist of matrices in $\UT(n,p)$ which are zero on the diagonal just above the main diagonal. 
All of these facts are proved in \cite{suzuki}.

There has been extensive study of the conjugacy classes and characters of $\UT(n,p)$. These have resisted explicit description and 
form a well known ``wild'' problem \cite{wild}. 
 See \cite{p-s}, \cite{isaacs} for reviews.  This difficulty gave rise to Carlos Andr\'e's 
`super-character theory'. This lumps together certain conjugacy classes into superclasses and sums certain irreducible characters forming 
super-characters. These are constant on superclasses and decompose the regular character. They have an elegant description involving 
set partitions. For surveys see \cite{aguiar}, \cite{d-i}.  We began this project hoping to understand what made this character theory so 
complicated. Asking whether or not it was reasonable to crudely equate `complicated' with `random' led us to the findings in sections 1-4 above. 
This section shows how the `randomness' of the Heisenberg group manifests itself in $\UT(n,p)$.

Recall that, in our notation, the subscript in $H_{2n+1}$ counts the maximal possible number of nonzero nondiagonal entries, whereas the 
$n+2$ in $\UT(n+2,p)$ refers to the dimension of the matrix. In what follows we will often abbreviate $\UT(n,p)$ as ``$U_n$''. 
Summarizing, in our notation, elements of  
$H_{2n+1}$ and $U_{n+2}$ are both uni-upper-triangular matrices of the same size.   

\begin{lemma} \label{l:normal}
\[ H_{2n+1} \mbox{ is a normal subgroup of } U_{n+2}. \]
\end{lemma}

\begin{proof}[First proof]
Define a map from $U_{n+2}$ to $U_n$ 
by simply erasing the top row, bottom row, first column and last column of a matrix in 
$U_n$. By inspection, this is a homomorphism onto $U_n$ 
with kernel $H_{2n+1}(p)$.
\end{proof}

\begin{proof}[Second proof, using additional information about superclasses]
First, $H_{2n+1}$ is a subgroup of $U_{n+2}$. So it would suffice to show that $H_{2n+1}$ is a union of superclasses (because then $H_{2n+1}$ is a union of 
conjugacy classes). But if $\mathcal{C}$ is a superclass which contains some element of $H_{2n+1}$,  
then all its elements must be elements of $H_{2n+1}$, because of the 
form of the canonical representative. 
\end{proof}

Iterating this mapping, going from $U_n$ to $U_{n-2}$ 
and so on gives a way of picturing $U_{n+2}$ 
as a tower of Heisenberg groups, each normal in the next.  Of course, $p$-groups have many such series; this one can be refined so that each 
subgroup has index $p$ in the one above. These series can be hard to describe and work with. One of our discoveries is that these 
nested Heisenberg representations are quite understandable. 

Nonetheless, the conjugacy classes of $H_{2n+1}$ in $U_{n+2}$ will often be bigger than they were in $H_{2n+1}$. 
To build an explicit example it is useful to use the fact that sometimes the neat superclasses described above are actually 
conjugacy classes in $U_n$. 
Carlos Andre \cite{andre} has a necessary and sufficient condition for this to happen 
(thanks to Nat Thiem for spelling this out for us). The following construction uses this. Suppose  
 $1 < k,\ell < n+2$. 
Let $C \subseteq H_{2n+1} \subseteq U_{n+2}$ be the set of all matrices $A$ such that (1) 
$a_{1,\ell} \neq 0$ and $a_{1,i} = 0$ for $1 < i < \ell$, and (2) $a_{k,n} \neq 0$ and 
$a_{j,n} = 0$ for $k < j < n$.  Then $C$ is a conjugacy class of $U_{n+2}$ if $k < \ell$, and 
a union of more than one conjugacy class otherwise.

\begin{obs} \label{o:36}
The image of $U_{n+2}$ under the quotient by $H_{2n+1}$ is $U_n$.   Thus $U_{n+2}$ is an extension of the normal subgroup 
$H_{2n+1}$ by the quotient $U_{n}$. In fact, this extension is a semi-direct product. The matrices in 
$U_{n+2}$ which are zero in the top row and last column form a subgroup isomorphic to $U_{n}$, and the intersection with the subgroup which is 
non-zero only in the top row and last column $($a copy of $H_{2n+1}$$)$ is the identity, so the extension splits. This allows picturing 
$U_{n}$ inside $U_{n+2}$ which will become important in Theorem $\ref{t:325}$ below. 
\end{obs}

\begin{cor} \label{n:h} 
Each $U_{n+2}$ contains a normal subgroup of index $p^{n(n-1)/2}$ whose $\Gamma$ is a quasirandom graph with density $\approx \frac{1}{p}$ 
$($as $n \rightarrow \infty$$)$. 
\end{cor}

In \ref{n:h} 
the size of the subgroup is fairly small, and a priori it 
need not say much about the global structure of $\Gamma(U_{n+2})$, but 
again, note $U_n$ is the quotient of $U_{n+2}$ by precisely this subgroup.    

The following introduces notation for this iterated decomposition that is 
used throughout the rest of this section. For simplicity in these iterated constructions, we will generally assume $n$ is odd. 
Definition \ref{d:maps} is illustrated in Figures 2 and 3. 

\begin{defn} \label{d:maps}
For $X \in U_{n+2}$, 
\begin{enumerate}
\item[(a)] let 
$u(X) \in U_n$ be the image of $X$ under the quotient of $U_{n+2}$ by $H_{2n+1}$, i.e., 
the unique $A \in U_n$ such that $a_{k,\ell} = x_{k-1, \ell-1}$ for $2 \leq k < \ell \leq n+1$;  
\item[(b)] let $h(X) \in H_{2n+1}$ be the unique $B \in H_{2n+1}$ such that $b_{1,k} = x_{1,k}$ for all $1 \leq k \leq n+2$ and 
$b_{\ell, n+2} = x_{\ell, n+2}$ for all $1 \leq \ell \leq n+2$. 
\end{enumerate} 
\end{defn} 

\noindent 
Observe that 
\[ X \mapsto ( u(X), h(X) ) \in U_{n} \times H_{2n+1} \]
defines a bijection from $U_{n+2}$ to $U_n \times H_{2n+1}$.  Informally, $u(X)$ is the ``core'' of $X$, the matrix resulting by 
removing the outermost rows and columns, and 
$h(X)$ is its ``shell.'' The first operation, $u$, can be iterated by defining $u^0(X) = X$ and $u^1(X) = u(X)$ and 
$u^2(X) = u(u(X)) \in U_{n-2}$, and so on down to $[1] \in U_1$. 

\br
\begin{center}

\includegraphics[height=29mm]{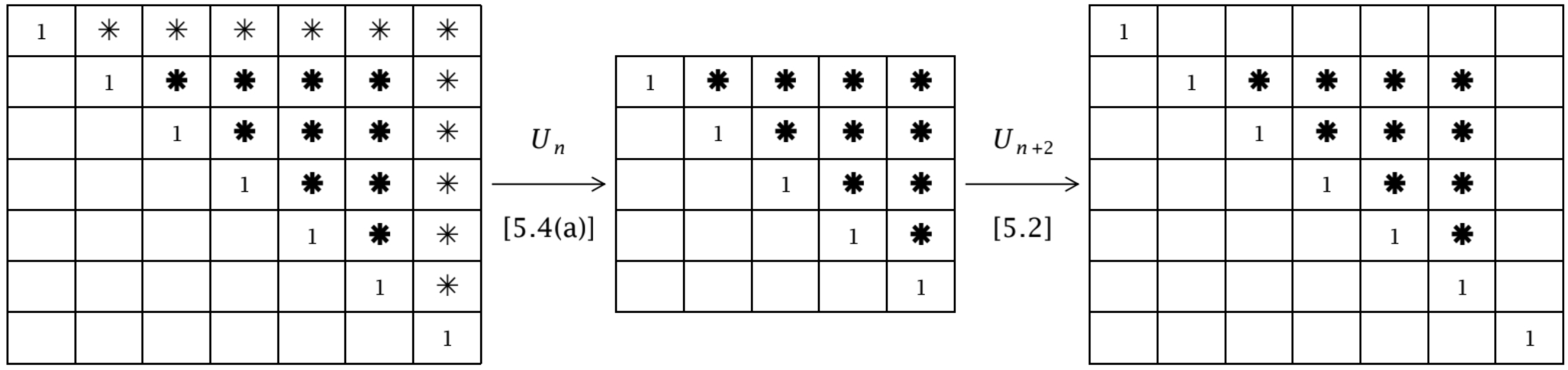}

\textcolor{gray}{\begin{small}\textsc{Figure 2}. $u(X)$ in $U_n$, see $\ref{d:maps}(a)$, and in $U_{n+2}$, see $\ref{o:36}$, for $n=5$.\end{small}}

\end{center}

\br

\begin{center}
\includegraphics[height=29mm]{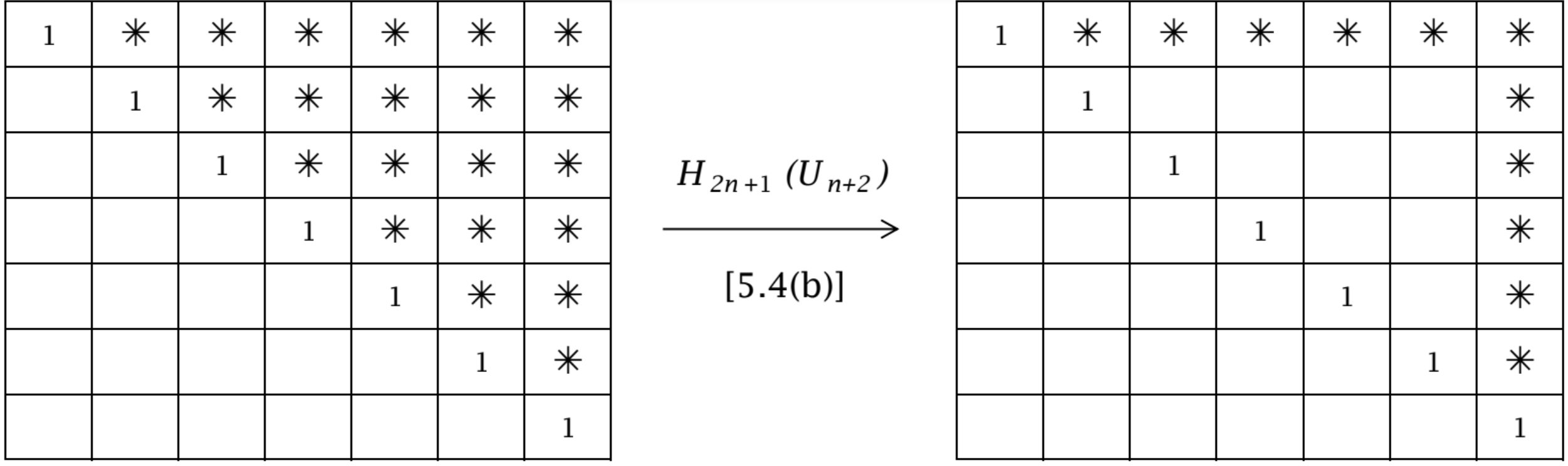}

\textcolor{gray}{\begin{small}\textsc{Figure 3}. 
$h(X)$ in $H_{2n+1}$ (or $U_{n+2}$), see $\ref{d:maps}(b)$, for $n=5$.\end{small}} 
\end{center}

\br
\begin{defn} \label{d:partial-order}
Define a partial order on $\uu_{n+2} = \bigcup \{ U_m : m\leq n+2, m \mbox{ odd} \} $ by: 
\[ A \tlf X \]
if $A = u^\ell(X)$ for some $0 \leq \ell \leq \frac{n+1}{2}$. 
\end{defn}

In the next observation, and in what follows, we think of trees rooted at the bottom and growing upwards.  We will 
 give several observations and definitions before pausing to summarize the picture. 

\begin{obs} Under this partial order
$(\uu_{n+2}, \tlf)$ 
is a tree with root $[1]$ and with uniform branching at each level: for $m$ odd and $m \leq n+2$, for each $A \in U_m$, 
\[ | \{ X : u(X) = A \} | = | \{ X : A \tlf X \} \cap U_{m+2} | = |H_{2m+1}| = p^{2m+1}. \]
\end{obs}

\begin{obs} \label{o:360} A necessary condition for $X, Y$ to commute in $U_{n+2}$ is that: 
\\ $($$u(X), u(Y)$ commute in $U_n$$) \land ($$h(X), h(Y)$ commute in $H_{2n+1}$$)$.   Thus:  
\begin{enumerate}
\item[(a)] if $X, Y$ commute in $U_{n+2}$, then $u^\ell(X)$, $u^\ell(Y)$ commute in $U_{n-2\ell}$. 

\item[(b)] for any $A, B \in U_n$, the bipartite commuting graph 
\[ \Gamma\left( \{ X \in U_{n+2}  : u(X) = A \} \times \{ Y \in U_{n+2} : u(Y) = B \} \right) \]
is a subgraph of $\Gamma(H_{2n+1} \times H_{2n+1})$ $($a priori, possibly empty or not induced$)$.
\end{enumerate}
\end{obs} 

\noindent The next definition asks ``how many descendants of $B$ (in the sense of the partial order $\tlf$) 
commute with $X$.''  Note that while $X \in U_{n+2}$, here $B$ belongs to $\uu_{n+2}$ from \ref{d:partial-order} above, so a priori may be from any 
$U_m$ with $1 \leq m \leq n+2$, $m$ odd. Recall also that for simplicity $n$, hence $n+2$, is odd.

\begin{defn}
Given $X \in U_{n+2}$ and $B \in \uu_{n+2}$, define ``the degree of $X$ over $B$'':
\[ \deg(X, B) = | \{ Y \in U_{n+2} :  B = u^\ell(Y) \mbox{ for some $\ell$, and $Y$ commutes with $X$} \} |. \] 
\end{defn}

\begin{expl}
If $X, B \in U_{n+2}$, then $\deg(X, B) = 1$ if and only if 
$X$ and $B$ commute, otherwise it is zero. 
If $X \in U_{n+2}$ and $B \in U_{n-\ell k}$, then unless $B$ commutes with $u^\ell(X)$, $\deg(X, B) = 0$. 
If $B = [1] \in U_1$, then $\deg(X,B)$ is simply the degree of $X$ in the 
graph $\Gamma(U_{n+2})$ $($or $\Gamma(U_{n+2} \times U_{n+2})$$)$.  
\end{expl}

\begin{disc}
\emph{We pause to reflect on the picture this sequence of definitions gives 
of the structure of $\Gamma(U_{n+2})$.  (Since in $U$ the conjugacy classes are no longer 
of uniform size, it is at least initially more convenient to work with $\Gamma$ and not $\tilde{\Gamma}$.) 
The picture appears by induction. 
Since $U_3 = H_3$, $\Gamma(U_3)$ is described by the previous sections. Let $n$ be odd and greater than or equal to 3.  
Reversing the map from the proof of Lemma \ref{l:normal} amounts to blowing up each vertex $A$ of $\Gamma(U_n)$ to a set
$\{ X \in U_{n+2} : u(X) = A \}$ of size 
$p^{2n+1} = |H_{2n+1}|$.  To form $\Gamma(U_{n+2})$ we need to see how to put edges among these vertices. 
If $A, B \in \Gamma(U_n)$ are not connected by an edge, then 
in $\Gamma(U_{n+2})$, the bipartite commuting graph between $\{ X \in U_{n+2} : u(X) = A \}$ and $\{Y \in U_{n+2} : u(Y) = B \}$, call it 
$G(A,B)$ for this discussion, 
is empty. This is the content of saying: that $u(X), u(Y)$ commute is necessary for $X$ and $Y$ to commute. 
If $A, B$ are connected by an edge, then $G(A, B)$ is a subgraph 
of $\Gamma(H_{2n+1} \times H_{2n+1})$ under the identification $(X, Y) \mapsto (h(X), h(Y))$.  ``Subgraph'' is the content of 
saying: that $h(X), h(Y)$ commute is necessary, but not always sufficient, for $X$ and $Y$ to commute. 
An example where $G(A,B)$ is \emph{exactly} $\Gamma(H_{2n+1} \times H_{2n+1})$ is when $A = B = \id_{U_n}$; 
this is \ref{n:h} above.   A different situation occurs when $A = B$ has nonzero entries all along the diagonal 
just above the main diagonal and zeros otherwise. 
As can be calculated (e.g. using the equations below) 
the vertices of this $G(A, B)$ have degree $\leq p^3$ (independent of $n$). 
Meanwhile, a cumulative effect of this inductive ``sparsification'' is reflected in the density of $\Gamma(U_n)$ 
decreasing rapidly as $n$ grows, see below.  The more local irregularities in the rates of decrease suggested by this 
discussion seem worthy of study.} 
\end{disc}

We now work out a simple sufficient condition for $G(A,B)$ to contain a full $\Gamma(H_{2m+1} \times H_{2m+1})$ as an  \emph{induced} subgraph  
for some $m \leq n$ related naturally to $A$, $B$.  Of course, the condition assumes $A, B$ commute.  
Similar results can be obtained considering $k$-partite graphs above $A_1, \dots, A_k$ (provided they form a $k$-clique in $\Gamma(U_n)$). 
The reader may wish to look ahead to Discussion \ref{d:summary} or to the figures in the text.

Whether $X, Y \in U_{n+2}$ commute may be studied via the following elementary picture. 
Let $E_{i,j}$ be the equation asserting the $(i,j)$-entries of $XY$ and $YX$ are equal: 
$\langle$row $i$ of $X$$\rangle \cdot \langle$column $j$ of $Y \rangle = \langle$row $i$ of $Y$$\rangle \cdot \langle$\possible{column $j$ of $X$} $\rangle$. 
Consider $A = u(X)$,  
 its entries indexed from $(2,2)$ to $(n+1, n+1)$, and likewise for $B = u(Y)$.  
 Towards understanding $\Gamma(\{ X : u(X) = A \} \times \{ Y : u(Y) = B \})$, 
it is convenient to think of $A$, $B$ as fixed, so all the ``core'' entries of $X$ and $Y$ are determined, whereas the 
``shell'' entries $x_{1,2}, \dots, x_{1,n+2}, x_{2,n+2}, \dots, x_{n+1, n+2}$ of $X$ and 
$y_{1,2}, \dots, y_{1,n+2}, y_{2,n+2}, \dots, y_{n+1, n+2}$ of $Y$ are for now free variables. 
Then these equations $E_{i,j}$ are of three kinds:
\begin{enumerate}
\item[(a)] $E_{1,3}, \dots, E_{1,n+1}$ are $(n-1)$ equations in the $2(n-1)$ variables $\{ x_{1,2}, \dots, x_{1,n} \}$, $\{ y_{1,2}, \dots, y_{1,n} \}$. \footnote{for visualization, here are a few:   $[{E_{1,3}} :  ~x_{1,2} \cdot b_{2,3} = \cdots]$, 
$[ E_{1,4}:   ~x_{1,2} \cdot b_{2,4} + x_{1,3} \cdot b_{3,4} = \cdots]$,
$[ E_{1,5} :  ~x_{1,2} \cdot b_{2,5} + x_{1,3} \cdot b_{3,5} + x_{1,4} \cdot b_{4,5} = \cdots]$, 
$\dots$, $[E_{1,n+1}:   ~x_{1,2} \cdot b_{2,n+1} + x_{1,3} \cdot b_{3,n+1} 
+ \cdots + x_{1,n} \cdot b_{n,n+1} = \cdots]$ 
where ``$= \cdots$'' repeats the expression on the left replacing $x$'s by $y$'s of the same index, and 
$b$'s by $a$'s of the same index (some terms which cancel were omitted).} We omit $E_{1,2}$ because it is always satisfied. 
\item[(b)] $E_{2,n+2}, \dots, E_{n,n+2}$ are $(n-1)$ equations in the $2(n-1)$ variables [notice this set has empty intersection with the variables from (a)] 
$\{ x_{3, n+2}, \dots, x_{n+1, n+2} \}$, $\{ y_{3, n+2}, \dots, y_{n+1, n+2} \}$.  Likewise, omit $E_{n+1,n+2}$ which is always satisfied. 
\item[(c)] the ``Heisenberg'' equation $E_{1,n+2}$ asserts the top right corners are equal: 
\[ x_{1,2} y_{2,n+2} + \cdots + x_{1,n+1} y_{n+1, n+2} =  y_{1,2} x_{2,n+2} + \cdots + y_{1,n+1} x_{n+1, n+2}. \]
\end{enumerate}
Define $\sigma_A \subseteq \{ 3, \dots n+1 \}$ to be the set of columns with at least one nonzero $A$-entry above the diagonal,  
 $\tau_A \subseteq \{ 2, \dots, n \}$ the set of rows with at least one nonzero $A$-entry above the diagonal, and likewise for 
 $\sigma_B, \tau_B$.\footnote{Although the sets $\sigma_A, \tau_A$ depend only on $A$, they are computed by 
 considering the indices entries of $A$ would have sitting inside a matrix of $U_{n+2}$.} 

If $i \in \{ 3, \dots, n+1 \} \setminus \sigma_B$ then the left-hand side of 
$E_{1,i}$ is zero, and also $y_{i,n+2}$ has zero coefficients [so effectively does not appear] in (b).
If $i \in \{ 3, \dots, n+1 \} \setminus \sigma_A$ then the right-hand side of 
$E_{1,i}$ is zero, and $x_{i,n+2}$ has zero coefficients in (b). 
Likewise, if  $j \in \{ 2, \dots, n \} \setminus \tau_B$ then the left-hand side of $E_{j,n+2}$ 
is zero and $y_{1,j}$ has zero coefficients in (a), 
and if $j \in \{ 2, \dots, n \} \setminus \tau_A$ then 
the right-hand side of $E_{j,n+2}$ is zero and $x_{1,j}$ has zero coefficients in (a). 

It follows that given tuples (of elements of $\mathbb{F}_p$)  
$\langle \mx_{1,i} : i \in \tau_B \rangle ^\smallfrown \langle \mx_{j,n+2} : j \in \sigma_B \rangle$ and 
$\langle \my_{1,i} : i \in \tau_A \rangle ^\smallfrown \langle \my_{j,n+2} : j \in \sigma_A \rangle$ which together solve (a) and (b), if we 
subsequently restrict to $\mcx = \{ X : u(X) = A $, $x_{1,i} = \mx_{1,i}$ for $i \in \tau_B$, $x_{j,n+2} = \mx_{j,n+2}$ for $j \in \sigma_B \}$ and
$\mcy = \{ Y : u(Y) = B $, $y_{1,i} = \my_{1,i}$ for $i \in \tau_A$, $y_{j,n+2} = \my_{j,n+2}$ for $j \in \sigma_A \}$ 
then the commuting between $\mcx$ and $\mcy$ is controlled only by the Heisenberg equation, and so corresponds to 
an induced subgraph of $\Gamma(H_{2n+1} \times H_{2n+1})$.

\begin{center}
\includegraphics[width=40mm]{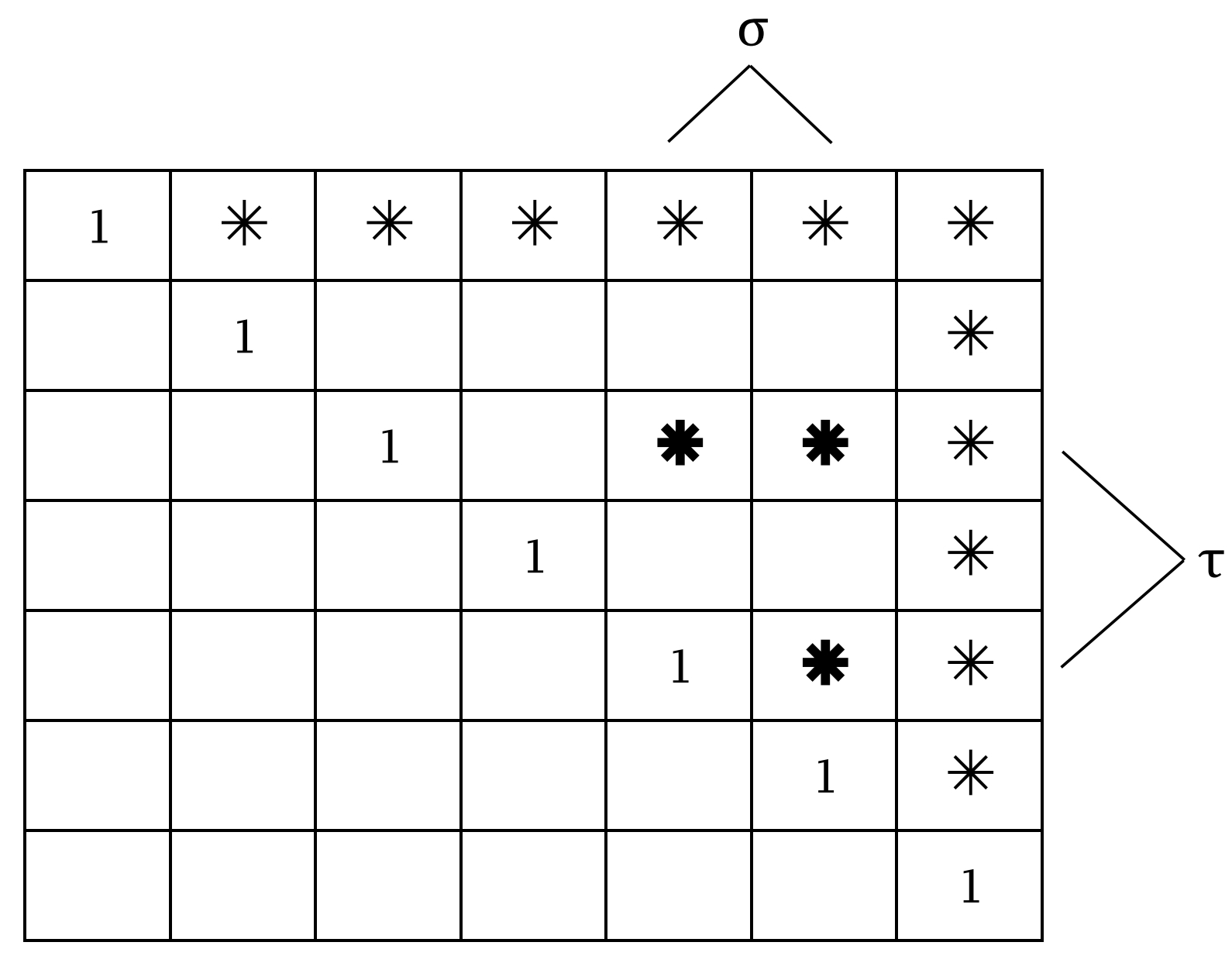}

\begin{small}\textcolor{gray}{\textsc{Figure 4}. 
An example for $n=5$. Here $A = u(X)$ is nonzero \\ at values 
$a_{3,5}$, $a_{3,6}$, $a_{5,6}$, so $\sigma_A = \{ 5, 6 \}, \tau_A = \{ 3, 5 \}$.}
\end{small} 
\end{center}

\br

\begin{disc} \label{d:summary}
\emph{Here is an informal summary of where we are. In order for $X, Y \in U_{n+2}$ to commute, it is necessary and sufficient that the 
following conditions are satisfied. First, $u(X)$ must commute with $u(Y)$ in $U_n$. Second, the equations asserting that 
the elements of $XY$ and $YX$ along the top row and right column, i.e. in positions $(1,2), \dots, (1,n+1)$, 
$(2,n+2), \dots, (n+1, n+2)$, are equal, must hold.  
Third, the ``Heisenberg equation,'' asserting that the top right corners of $XY$ and $YX$ 
are equal (i.e. that $h(X)$ commutes with $h(Y)$) must hold.  Restating, whenever $u(X), u(Y)$ commute in $U_n$ 
and the second equations are satisfied, 
control of commuting falls entirely to the 
Heisenberg equation. As for where we are going: as we will see below, if $A, B \in U_n$ (or some tuple $A_1, \dots, A_k \in U_n$) all commute and 
have ``enough'' zeros, then there exist reasonably large subsets 
$\mcx \subseteq \{ X : u(X) = A \}$ and $\mcy \subseteq \{ Y : u(Y) = B \}$  
between which commuting is entirely controlled by the Heisenberg equation. (``Reasonably large'': 
because the non-Heisenberg equations can then be satisfied by fixing a reasonably small number of values in the top row and right column.)
If moreover these reasonably large subsets are ``symmetric'' (see below)  then under the maps 
$X \mapsto h(X)$, $Y \mapsto h(Y)$, our $\mcx, \mcy$ both map to the same copy of some  
$H_{2m+1}$ inside $H_{2n+1}$,  or the parallel for the $k$-partite case.  
A notion of symmetrization is introduced for this reason.}
\end{disc}

\begin{defn} \label{d:sym}
Given $\sigma \subseteq \{ 3, \dots, n+1 \}$, $\tau \subseteq \{ 2, \dots, n \}$ define the \emph{symmetrizations} 
$\sigma_* = \tau_* = \sigma \cup \tau \subseteq \{ 2, \dots, n+1 \}$. 
\end{defn}

The next theorem summarizes an instance of this analysis, showing the recurrent appearance of
\possible{$\Gamma(H)$ (for some $H$) in $\Gamma(U)$}. 
When $t = 1$, we find a copy of $\Gamma(H)$ for some $H$; 
when $t=2$, of $\Gamma(H \times H)$; we can just as easily state a more general version for arbitrary finite $t$, finding $\Gamma(H \times \cdots \times H)$. 
It is illustrated in Figure 5 below. 

\begin{theorem} \label{t:325}  
Suppose $A_1, \dots, A_t$ form a clique in $\Gamma(U_n)$. 
Then there are subsets $\mathbf{X}^*_i \subseteq \{ X : u(X) = A_i \}$ 
for $i = 1, \dots, t$ such that the $t$-partite graph whose pieces are $\mathbf{X}^*_i$ $($in $\Gamma(U_{n+2})$$)$
is naturally isomorphic to the $t$-partite commuting graph whose pieces are $H_{2m+1}$, where  
$m = n - | \bigcup_i \sigma_{A_i} \cup \bigcup_i \tau_{A_i} |$. 
\end{theorem}

\begin{proof}
The above analysis asked for tuples of elements of $\mathbb{F}_p$ with indices in a restricted set solving (a) and (b). 
Notice that in this case the trivial sequences 
(all zeros) are always a solution. Suppose, then, that we are given $A_1, \dots, A_t \in U_n$ forming a clique in $\Gamma(U_n)$ and 
we are considering the $t$-partite commuting graph whose $i$-th part is $\{ X : u(X) = A_i \}$. 
Letting $\sigma_* = \tau_* = \bigcup_i \sigma_{A_i} \cup \bigcup_i \tau_{A_i}$
and restricting $\{ X : u(X) = A_i \}$ to $\mcx^*_i = \{ X : u(X) = A_i $, $x_{1,i} = 0$ for $i \in \sigma_*$, $x_{j,n+2} = 0$ for $j \in \tau_* \}$, 
the commuting on the $t$-partite graph $\mcx^*_1, \dots, \mcx^*_t$ is entirely controlled by the Heisenberg equation; 
moreover, there is a particular copy of $H_{2m+1} \subseteq H_{2n+1}$ which is the common image of 
every $\mcx^*_i$ under the map $X \mapsto h(X)$. 
\end{proof}

\begin{disc}
\emph{First, regarding symmetrization: had we said in \ref{t:325} only that the resulting graph was  
an induced subgraph of $\Gamma(H_{2n+1} \times \cdots \times H_{2n+1})$,  a priori the vertex sets $V_i \subseteq H_{2n+1}$ 
would need not be the same; for small sizes, this may affect statistics.   This said, no real effort is being made here to optimize $m$.  
Second, some readers of Theorem \ref{t:325} may see the possible power of $H$ to recur, especially when the $A_i$'s are similar or small, 
and some may see the (a priori) possible power of 
other $A_i$'s to keep $m$ low by collectively spreading out.  Both forces are 
interesting in their own right (see \ref{f:32} below).  Further information on their 
relative strengths, perhaps on average, could certainly be interesting.} 
\end{disc}

\br
\begin{center}
\includegraphics[height=28mm]{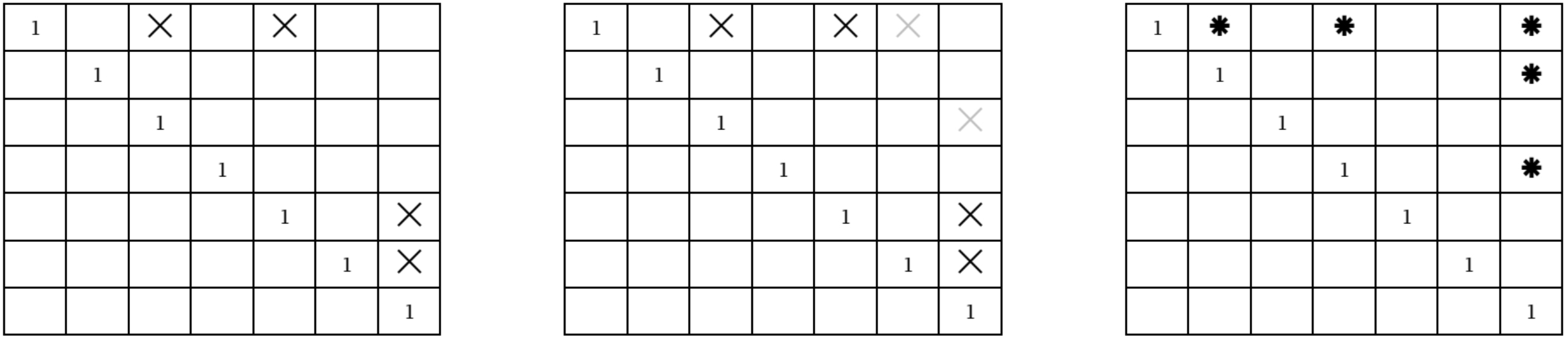}

\begin{small}\textcolor{gray}{\textsc{Figure 5}. 
Theorem \ref{t:325} for $A_1 = \dots = A_t = A$ as in figure 4.  
\\ First, the left image marks values implicated by $\tau$ (on the top row) 
\\ and $\sigma$ (right column).  The middle image symmetrizes. Now 
set all \\ marked values e.g. to zero, and on the right is the copy of 
$H_{2k+1}$ \\ which remains. Here $k=2$, $n=5$.}\end{small}
\end{center}
\br

Here is a simple example of a corollary of the previous proof, however, notice that the partition given depends on both $A$ and $B$; 
it is not claimed that a single partition of $\{ X : u(X) = A \}$ works against $\{ Y : u(Y) = B \}$ for any $B$.  

\begin{cor} \label{c:part}
Let $A, B \in U_n$ and $m = n - |\sigma_A \cup \sigma_B \cup \tau_A \cup \tau_B|$. 
There are equipartitions of $\mcx = \{ X : u(X) = A \}$ 
into $\{ \mcx_i : i < p^{2(n-m)} \}$  
and of $\mcy = \{ Y : u(Y) = B \}$
into $\{ \mcy_i : i < p^{2(n-m)} \}$ 
such that: 
for any two $i,j$, we have that $\Gamma(\mcx_i \times \mcy_j)$ is either isomorphic 
to $\Gamma(H_{2m+1} \times H_{2m+1})$ or the empty graph. 
\end{cor}

\begin{proof} 
Let $\sigma_* = \tau_* = \sigma_A \cup \sigma_B \cup \tau_A \cup \tau_B$.  
Partition the elements of $\{ X : u(X) = A \}$ according to their values on $\{ x_{1,i} : i \in \tau_* \} \cup \{ y_{j,n+2} : j \in \sigma_* \}$. 
If $A, B$ do not commute, there is an empty graph in each case. Otherwise, 
fixing $i,j$ we have that the equations (a) and (b) [as above, or in \ref{t:325}] are uniformly satisfied or unsatisfied on all pairs 
from $\mcx_i$ and $\mcy_j$. The second case gives the empty graph, and the first reduces to the Heisenberg equation 
on the values which remain free. 
\end{proof}

Similar arguments show other ways quasirandomness is pervasive in $\Gamma(U)$. 

\begin{disc} 
\emph{In $H$ the randomness is only one level deep. In $U$ 
the appearance of quasirandomness is 
indestructible in the sense that it recurs after arbitrary quotients by $H$, until we reach $[1]$ (or in the other direction, reappears  
under unlimited reverse quotienting).}
\end{disc}

\begin{disc}
\emph{Still, 
if $A, B$ commute in $U_{n}$, then 
necessarily $u^\ell(A)$ commutes with $u^\ell(B)$ for $0 \leq \ell < \frac{n+1}{2}$.  
Moreover, $m$ plays a role: there are fewer edges between $\{ X : u(X) = A \}$ and $\{ Y : u(Y) = B \}$ when 
$A$, $B$ provide more constraints. 
Is there a quantifiable drift of edges as $n \rightarrow \infty$ 
towards having both ``lower rank'' endpoints, visible for instance in $\deg(X, [1])$ in terms of $\sigma, \tau$ of $u(X)$?} 
\end{disc}

\begin{disc}
\emph{The successive peeling off of Heisenberg groups can be organized differently. By simply multiplying matrices it is easy to see that in 
$U(n+2,p)$ the subgroup $H(2n+1,j,p)$ which is non-zero only in the top $j$ rows and last $j$ columns is a normal subgroup, with complement 
$U(j)$ consisting of matrices in $U(n+2,p)$ which are zero in the top $j$ rows and last $j$ columns. This complement is isomorphic to 
$U(n-2j,p)$ $($and the extension splits$)$. $H(2n+1,j)$ is a subgroup of $H(2n+1,j+1,p)$ for $j$ from $1$ to $n/2$ with the top being 
$U(n+2,p)$.  It is not always the case that there is a complement of a normal pattern subgroup: the center and commutator of $U(n+2,p)$ are 
normal pattern subgroups without complements. For more on this see \cite{marberg}.}
\end{disc}

\begin{disc}
\emph{There has been recent work in developing various supercharacter theories for groups such as $H(\omega)$ and some $U(\omega)$. For an elegant 
exposition, see \cite{lochon}. It would be fascinating to see some applications of this theory along the lines of \cite{AC}. See the work of 
 \cite{BSC} for first steps.} 
\end{disc}

We conclude with several comments about this very interesting graph. 

\begin{disc} \label{f:32}
\emph{\cite[p. 22]{p-s} gives the known bounds for $c(U_n(p))$ as:} 
\[ p^{\frac{n^2}{12}(1 + o(1))} \leq c(U_n(p)) \leq  p^{\frac{7 n^2}{44}(1 + o(1))} \]
\emph{where the lower bound is attributed to Higman and the upper bound to Soffer. } 
\end{disc}

By Fact \ref{et-fact} 
the average degree of a vertex in $\Gamma(U_n)$ is subject to essentially the same bounds.
It also follows that as $n \rightarrow \infty$ the density of $\Gamma(U_n)$ is ``super-sparse'' of order $p^{-cn^2}$. 
There is a lot of effort in the graph theory limit community to define structure theory for sparse graphs. 
These examples show that truly super-sparse graphs can still have interesting structure.

\br


\end{document}